\documentclass[letterpaper,11pt,reqno]{amsart}
\usepackage[utf8]{inputenc}
\usepackage{amsmath}
\usepackage{amssymb}
\usepackage{amsthm}
\usepackage{hyperref}
\usepackage[arrow, matrix, curve]{xy}
\usepackage{color}
\usepackage{enumitem}

\DeclareMathOperator{\GL}{GL}

\DeclareMathOperator{\SL}{SL}

\DeclareMathOperator{\Aut}{Aut}

\newcommand{\frakg}{\mathfrak{g}}

\newcommand{\frakk}{\mathfrak{k}}
\newcommand{\frakl}{\mathfrak{l}}

\newcommand{\frakp}{\mathfrak{p}}

\newcommand{\BB}{\mathbb{B}}
\newcommand{\CC}{\mathbb{C}}

\newcommand{\RR}{\mathbb{R}}

\newcommand{\ZZ}{\mathbb{Z}}

\newcommand{\calF}{\mathcal{F}}
\newcommand{\calH}{\mathcal{H}}
\newcommand{\calI}{\mathcal{I}}

\newcommand{\calK}{\mathcal{K}}
\newcommand{\calL}{\mathcal{L}}
\newcommand{\calO}{\mathcal{O}}
\newcommand{\calP}{\mathcal{P}}

\newcommand{\calS}{\mathcal{S}}

\newcommand{\calW}{\mathcal{W}}

\DeclareMathOperator{\tr}{tr}
\DeclareMathOperator{\Tr}{Tr}
\DeclareMathOperator{\Det}{Det}

\DeclareMathOperator{\End}{End}

\DeclareMathOperator{\id}{id}

\DeclareMathOperator{\rk}{rk}

\theoremstyle{plain}
\newtheorem{theorem}{Theorem}[section]
\newtheorem{proposition}[theorem]{Proposition}

\newtheorem{corollary}[theorem]{Corollary}

\theoremstyle{definition}

\newtheorem{remark}[theorem]{Remark}

\numberwithin{equation}{section}

\title[Generalized Laguerre functions and Whittaker vectors]{Generalized Laguerre functions and Whittaker vectors for holomorphic discrete series}

\author{Jan Frahm}
\address{Department of Mathematics, Aarhus University, Ny Munkegade 118, 8000 Aarhus C, Denmark}
\email{frahm@math.au.dk}

\author{Gestur \'{O}lafsson}
\address{Department of Mathematics, Louisiana State University, Baton Rouge, LA 70803, USA}
\email{olafsson@math.lsu.edu}

\author{Bent {\O}rsted}
\address{Department of Mathematics, Aarhus University, Ny Munkegade 118, 8000 Aarhus C, Denmark}
\email{orsted@math.au.dk}

\begin{document}
\thanks{Part of the research in this paper was carried out within the online research community on Representation Theory and Noncommutative Geometry sponsored by the American Institute of Mathematics.
The first author was partially supported by a research grant from the Villum Foundation (Grant No. 00025373).
The second author was partially supported by Simons grant 586106.}



\maketitle

\begin{abstract}
We study degenerate Whittaker vectors in scalar type holomorphic discrete series representations of tube type Hermitian 
Lie groups and their analytic continuation. In four different realizations, the bounded domain picture, the tube 
domain picture, the $L^2$-model and the Fock model, we find their explicit $K$-type expansions. 
The coefficients are expressed in terms of the generalized Laguerre functions on the corresponding 
symmetric cone, and we relate the $K$-type expansions to the formula for the generating function of the
 Laguerre polynomials and to their recurrence relations.
\end{abstract}


\section*{Introduction}

For a Lie group $G$, the study of its unitary representations $\pi: G \rightarrow U({\calH})$ has roughly
speaking two aspects, namely the question of (1) classifying the irreducible ones up to equivalence, and
(2) finding models of these, including concrete Hilbert spaces ${\calH}$ and unitary actions of $G$ here.
For (1), great progress has been achieved, in particular for reductive groups $G$, and when
$G$ is of Hermitian type, the work of Harish-Chandra on the holomorphic discrete series gave
very explicit Hilbert spaces ${\calH}$ of holomorphic sections of vector bundles over the
corresponding Riemannian symmetric space $G/K$. Such spaces
are reproducing kernel Hilbert spaces, and together with related $L^2$-spaces they are among the
most natural concrete Hilbert spaces. The question (2) often relates to classical harmonic analysis
in the sense of involving classical special functions, and it is this aspect which we shall study in
the present paper, continuing our previous work \cite{FOO22}. Hence, we shall work with
four different models of the scalar holomorphic discrete series of groups of tube type, and
study degenerate Whittaker vectors in each, and also the integral operators exhibiting the intertwining
operators; in particular, we find explicit expansions of the Whittaker vectors in terms of orthonormal bases of $K$-types in each ${\calH}$. This is in the same spirit as the celebrated Segal-Bargmann transform, where the corresponding
integral kernel is essentially the generating function for the Hermite functions. In the
same way, the Whittaker vectors we find play a role of generating functions of generalized Laguerre polynomials on symmetric cones.

The starting point of this work was indeed the formula for the generating function of the classical Laguerre polynomials $L_n^\alpha(x)$, or rather for the corresponding Laguerre functions $\ell_n^\alpha(x)=e^{-x}L_n^\alpha(2x)$:
\begin{equation}
	\sum_{n=0}^\infty\ell_n^\alpha(x)t^n = (1-t)^{-\alpha-1}e^{-x\frac{1+t}{1-t}}.\label{eq:GenFctClassicalLaguerre}
\end{equation}
This formula appeared in a paper of Kostant~\cite[formula (5.15)]{Kos00} in the context of the analytic continuation of the holomorphic discrete series of the universal covering group of $\SL(2,\RR)$. (Note the slightly different normalization of the Laguerre polynomials in \cite{Kos00}, see Remark~\ref{rem:DefLaguerre} for our notation which seems to be standard in the mathematics literature, for instance in \cite[Chapter 6.2]{AAR99}.) More precisely, Kostant considers a realization of these representations $\pi_\nu$ on $L^2(\RR_+)$ and uses \eqref{eq:GenFctClassicalLaguerre} to show that the Laguerre functions are eigenfunctions of the Hankel transform, and then concludes that the Hankel transform is essentially the action of a certain group element in the representation.

On the other hand, the right hand side of \eqref{eq:GenFctClassicalLaguerre} has a striking similarity with the Whittaker vectors we found in \cite{FOO22} in a different realization of the analytic continuation of the representation $\pi_\nu$, namely the bounded domain realization on a space $\calH^2_\nu(D)$ of holomorphic functions on the unit disc $D\subseteq\CC$. Although Kostant also studies Whittaker vectors on $\pi_\nu$ in his paper \cite{Kos00}, he does not connect the formula \eqref{eq:GenFctClassicalLaguerre} to Whittaker vectors on $\calH^2_\nu(D)$. In this paper, we obtain such a relation in the more general setting of scalar type unitary highest weight representations of tube type Hermitian groups.\\

Let $G$ be a connected simple Hermitian Lie group. The scalar type unitary highest weight representations $\pi_\nu$ of the universal cover of $G$ ($\nu$ a complex parameter) can be realized in various different ways (see Section~\ref{sec:ScalarHolDS} for details). The most prominent realization is on a weighted Bergmann space $\calH_\nu^2(D)$ of holomorphic functions on the associated bounded symmetric domain $D\simeq G/K$, where $K\subseteq G$ is a maximal compact subgroup. If $G$ is of tube type, there is a Cayley transform that gives a biholomorphic equivalence between $D$ and a tube domain $T_\Omega=V+i\Omega\subseteq V_\CC$, where $V$ is a real vector space and $\Omega\subseteq V$ a symmetric cone. Hence, $\calH_\nu^2(D)\simeq\calH_\nu^2(T_\Omega)$, and we obtain a realization of $\pi_\nu$ on $\calH_\nu^2(T_\Omega)$. The vector space $V$ has a natural structure of a Euclidean Jordan algebra and the cone $\Omega$ can be viewed as the Riemannian symmetric space $L/(K\cap L)$ of the Levi subgroup $L$ of the Siegel parabolic subgroup of $G$. Moreover, a generalization of the classical Laplace transform gives a unitary isomorphism between $\calH_\nu^2(T_\Omega)$ and a weighted $L^2$-space $L^2_\nu(\Omega)$ on the cone $\Omega$, thus providing yet another realization of $\pi_\nu$, the so-called \emph{$L^2$-model}. Finally, in a more recent work \cite{Moe13}, the first author found a generalization of the classical Segal--Bargmann transform, establishing an isomorphism between $L^2_\nu(\Omega)$ and a Fock space $\calF_\nu^2(V_\CC)$ of holomorphic functions on $V_\CC$.

The $K$-types in $\pi_\nu$ can be parameterized by ${\bf m}=(m_1,\ldots,m_r)\in\ZZ^r$ with $m_1\geq\ldots\geq m_r\geq0$, where $r=\rk(G/K)$. Each $K$-type contains a unique (up to scalar multiples) $(K\cap L)$-invariant vector. In the realization on $\calH_\nu^2(D)$, this vector is a spherical polynomial $\Phi_{\bf m}(w)$ on $V_\CC$, and in the realization on $L^2_\nu(\Omega)$ it is a generalized Laguerre function $\ell_{\bf m}^\nu(u)$ (see Sections~\ref{sec:SphericalPolys} and \ref{sec:GenLaguerre} for details). Note that for $G=\SL(2,\RR)$ we have $V=\RR$, $\Phi_{\bf m}(w)=w^{m_1}$ and $\ell_{\bf m}^\nu(u)=m_1!\cdot\ell_{m_1}^{\nu-1}(u)$, the classical Laguerre functions.

The formula \eqref{eq:GenFctClassicalLaguerre} has a generalization to this context:
\begin{equation}
	\sum_{{\bf m}\geq0}\frac{d_{\bf m}}{(\frac{n}{r})_{\bf m}}\Phi_{\bf m}(w)\ell_{\bf m}^\nu(u) = \Delta(e-w)^{-\nu}\int_{K\cap L} e^{-(ku|(e+w)(e-w)^{-1})}\,dk,\label{eq:GenFctLaguerreFctIntro}
\end{equation}
where $d_{\bf m}$ is the dimension of the $K$-type associated with ${\bf m}$, $(\frac{n}{r})_{\bf m}$ a multivariate Pochhammer symbol where $n=\dim V$, $\Delta(z)$ the Jordan algebra determinant on $V_\CC$, $(z|w)$ the Jordan algebra trace form and $z^{-1}$ the Jordan algebra inverse of $z\in V_\CC$ (see Section~\ref{sec:Preliminaries} for the precise definitions). This formula was stated in \cite[Exercise 3 in Chapter XV]{FK94}, and we give two different proofs of it in Section~\ref{sec:GenFctLaguerre}. The first proof follows the hints in \cite{FK94} and is more or less a direct computation (see Section~\ref{sec:FirstProof}), and the second proof uses the intertwining operator $L^2_\nu(\Omega)\to\calH_\nu^2(D)$ given by composing Laplace and Cayley transform and mapping $\ell_{\bf m}^\nu$ to a multiple of $\Phi_{\bf m}$ (see Section~\ref{sec:SecondProof}). The second proof can also be applied to the Segal--Bargmann transform $L^2_\nu(\Omega)\to\calF_\nu^2(V_\CC)$ and yields new identities for the multivariate $I$-Bessel function on the cone $\Omega$ (see Section~\ref{sec:SecondProofVar}).

Now, for $u=te$, $t\in\RR$ and $e$ the identity element of the Jordan algebra $V$, the right hand side of \eqref{eq:GenFctLaguerreFctIntro}, as a function of $w\in D$, is a degenerate Whittaker vector in the realization $\calH_\nu^2(D)$ with respect to the unipotent radical $N$ of the Siegel parabolic subgroup. Hence, the left hand side of \eqref{eq:GenFctLaguerreFctIntro} can be viewed as the $K$-type expansion of the Whittaker vector. In Section~\ref{sec:KTypeExpansion} we provide a different proof of this expansion, starting with the expansion of the Whittaker vector in the $L^2$-model and mapping it to the other realizations, thus establishing $K$-type expansion formulas for Whittaker vectors in all four realizations. In all cases, the coefficients of this expansion contain the factor $\ell_{\bf m}^\nu(te)$, the special value of the generalized Laguerre functions at $u=te$.

Finally, in Section~\ref{sec:RecurrenceRelations} we explore the relation between the $K$-type expansion of Whittaker vectors and the recurrence relations for the Laguerre functions obtained in \cite{ADO06}. In rank $r=1$, the recurrence relations turn out to be equivalent to the Whittaker property for the expansion.

\section{Preliminaries}\label{sec:Preliminaries}

We recall the construction of Hermitian Lie groups of tube type via Euclidean Jordan algebras as well as several special functions on Euclidean Jordan algebras that play a role in the construction of different models of unitary highest weight representations. The main reference is \cite{FK94}, and for the last section also \cite{Moe13}.

\subsection{Euclidean Jordan algebras and related groups}

Let $V$ be a simple Euclidean Jordan algebra of dimension $n$ and rank $r$, $\Omega\subseteq V$ the symmetric cone of invertible squares and $T_\Omega=V+i\Omega\subseteq V_\CC$ the corresponding tube domain. The group $G=G(T_\Omega)$ of biholomorphic automorphisms of $T_\Omega$ is a simple Hermitian Lie group of tube type generated by the group of translations
$$ N = \{n_v:v\in V\}, \qquad \mbox{where }n_v:T_\Omega\to T_\Omega,\,z\mapsto z+v, $$
the group of linear automorphisms
$$ L = G(\Omega) = \{g\in\GL(V):g\Omega=\Omega\}, $$
and the inversion
$$ j:T_\Omega\to T_\Omega, \quad j(z)=-z^{-1}. $$

The map $\theta:G\to G,\,g\mapsto jgj^{-1}$ defines a Cartan involution on $G$ and we write $K=\{g\in G:\theta(g)=g\}$ for the corresponding maximal compact subgroup. Since $\theta$ leaves $L$ invariant, the intersection $K\cap L$ is a maximal compact subgroup of $L$, and it turns out that it equals the automorphism group of the Jordan algebra $V$:
$$ K\cap L = \Aut(V) = \{g\in G(\Omega):ge=e\}. $$

The tube domain $T_\Omega$ is biholomorphically equivalent to a bounded symmetric domain $D\subseteq V_\CC$, the biholomorphic isomorphism being the Cayley transform
$$ c:D\to T_\Omega, \quad c(w)=i(e+w)(e-w)^{-1} $$
with inverse
$$ p:T_\Omega\to D, \quad p(z)=(z-ie)(z+ie)^{-1}. $$
We can therefore also think of $G$ as the automorphism group of $D$.

We further denote by $L_\CC\subseteq\GL(V_\CC)$ the complex linear group with Lie algebra $\frakl_\CC$, the complexification of the Lie algebra $\frakl$ of $L$. It contains a maximal compact subgroup $U$ with Lie algebra $(\frakk\cap\frakl)+i(\frakp\cap\frakl)$, where $\frakg=\frakk\oplus\frakp$ is the Cartan decomposition of the Lie algebra $\frakg$ of $G$ with respect to the Cartan decomposition $\theta$. Then $L_\CC$ acts on $V_\CC$ with an open dense orbit $\{z\in V_\CC:\Delta(z)\neq0\}$, and every element in this orbit can be written as $z=ux$ with $u\in U$ and $x\in\Omega\subseteq V$.

\subsection{Spherical polynomials}\label{sec:SphericalPolys}

Denote by $L:V\to\End(V),\,L(x)y=xy$ the multiplication map and by $P:V\to\End(V),\,P(x)=2L(x)^2-L(x^2)$ the quadratic representation of $V$. The linear polynomial $\tr(x)=\frac{r}{n}\Tr L(x)$ on $V$ is called the \emph{Jordan trace}, and it gives rise to an inner product $(x|y)=\tr(xy)$ on $V$ which extends to a bilinear form $(z|w)$ on $V_\CC$. Further, there exists a unique polynomial $\Delta(x)$ of degree $r$ with $\Delta(e)=1$, the \emph{Jordan determinant}, such that
$$ \Delta(x)^{\frac{2n}{r}}=\Det P(x). $$
The Jordan determinant is a relative invariant for the action of the group $L$, i.e. there exists a character $\chi:L\to\RR_+$ such that
$$ \Delta(gx) = \chi(g)\Delta(x) \qquad (g\in L,x\in V). $$

Let $c_1,\ldots,c_r\in V$ be a maximal set of pairwise orthogonal primitive idempotents such that $c_1+\cdots+c_r=e$. If we denote by $V(c_j,\lambda)$ the eigenspace of $L(c_j)$ to the eigenvalue $\lambda$, then we have the \emph{Peirce decomposition}
$$ V=\bigoplus_{1\leq i\leq j\leq r}V_{ij}, $$
where $V_{ii}=V(c_i,1)=\RR c_i$ and $V_{ij}=V(c_i,\frac{1}{2})\cap V(c_j,\frac{1}{2})$ for $i<j$. The subspaces $V_{ij}$ for $i<j$ have a common dimension $d$, and we can write
$$ n = r + r(r-1)\frac{d}{2}. $$

For each $1\leq k\leq r$, the element $e_k=c_1+\cdots+c_k$ is an idempotent in $V$ and the eigenspace $V(e_k,1)$ of $L(e_k)$ to the eigenvalue $1$ is a Jordan algebra of rank $k$ with identity element $e_k$. Let 
\[\Delta_k(x)=\Delta_{V(e_k,1)}(P_kx),\]
where $\Delta_{V(e_k,1)}$ denotes the Jordan determinant of $V(e_k,1)$ and $P_k:V\to V(e_k,1)$ the orthogonal projection. Then $\Delta_k(x)$ is a polynomial on $V$ of degree $k$ with $\Delta_k(e)=1$, and $\Delta_r(x)=\Delta(x)$.

For ${\bf m}\in\ZZ^r$ we write ${\bf m}\geq0$ if
$$ m_1\geq m_2\geq\ldots\geq m_r\geq0. $$
In this case,
$$ \Delta_{\bf m}(x) = \Delta_1(x)^{m_1-m_2}\cdots\Delta_{r-1}(x)^{m_{r-1}-m_r}\Delta_r(x)^{m_r} $$
defines a polynomial on $V$ of degree $|{\bf m}|=m_1+\cdots+m_r$ with $\Delta_{\bf m}(e)=1$. Each of these polynomials $\Delta_{\bf m}(x)$ generates a subspace $\calP_{\bf m}(V)$ of polynomials on $V$ under the action of the group $V$ by $(g\cdot p)(x)=p(g^{-1}x)$. Let $d_{\bf m}$ denote the dimension of $\calP_{\bf m}(V)$. The spaces $\calP_{\bf m}(V)$ turn out to be irreducible representations of $L$, and they all contain a unique $(K\cap L)$-invariant polynomial $\Phi_{\bf m}(x)$ normalized such that $\Phi_{\bf m}(e)=1$. More precisely, we have the integral formula
$$ \Phi_{\bf m}(x) = \int_{K\cap L}\Delta_{\bf m}(kx)\,dk, $$
where $dk$ is the normalized Haar measure on $K\cap L$. The \emph{Hua--Kostant--Schmid decomposition} asserts that the space $\calP(V)$ of all polynomials on $V$ decomposes into the direct sum of all subspaces $\calP_{\bf m}(V)$.

The dimension $d_{\bf m}$ can be computed explicitly. By \cite[Chapter XIV.4]{FK94} we have
$$ d_{\bf m} = \frac{c(\rho)c(-\rho)}{c(\rho-{\bf m})c({\bf m}-\rho)}, $$
where
$$ c(\lambda) = \prod_{1\leq i<j\leq r}\frac{B(\lambda_j-\lambda_i,\frac{d}{2})}{B((j-i)\frac{d}{2},\frac{d}{2})} = \prod_{i<j}\frac{\Gamma(\lambda_j-\lambda_i)\Gamma((j-i+1)\frac{d}{2})}{\Gamma(\lambda_j-\lambda_i+\frac{d}{2})\Gamma((j-i)\frac{d}{2})} $$
is the $c$-function of the cone $\Omega$ and $\rho_j=(2j-r-1)\frac{d}{4}$. Hence,
\begin{align*}
	d_{\bf m} &= \prod_{i<j}\frac{((i-j+1)\frac{d}{2})_{m_j-m_i}(m_i-m_j+(j-i)\frac{d}{2})_{m_j-m_i}}{((i-j)\frac{d}{2})_{m_j-m_i}(m_i-m_j+(j-i+1)\frac{d}{2})_{m_j-m_i}}\\
	&= \prod_{i<j}\frac{((i-j)\frac{d}{2}+m_j-m_i)(m_j-m_i+(i-j-1)\frac{d}{2}+1)_{d-1}}{((i-j)\frac{d}{2})((i-j-1)\frac{d}{2}+1)_{d-1}}.
\end{align*}
In particular,
\begin{equation}
	\frac{d_{{\bf m}+e_j}}{d_{\bf m}} = \prod_{i\neq j}\frac{((i-j)\frac{d}{2}+m_j-m_i+1)(m_j-m_i+(i-j+1)\frac{d}{2})}{((i-j)\frac{d}{2}+m_j-m_i)(m_j-m_i+(i-j-1)\frac{d}{2}+1)}.\label{eq:QuotientDimensions}
\end{equation}

\subsection{The Gamma function}

For $\nu>\frac{n}{r}-1$ the integral
$$ \Gamma_\Omega(\nu) = \int_\Omega e^{-\tr(x)}\Delta(x)^{\nu-\frac{n}{r}}\,dx $$
converges absolutely and can be evaluated in terms of the classical Gamma function:
$$ \Gamma_\Omega(\nu) = (2\pi)^{\frac{n-r}{2}}\prod_{j=1}^r\Gamma\left(\nu-(j-1)\frac{d}{2}\right). $$
The Gamma function can be used to define the Pochhammer symbol
\begin{equation}
	(\nu)_{\bf m} = \frac{\Gamma_\Omega(\nu+{\bf m})}{\Gamma_\Omega(\nu)} = \prod_{j=1}^r\left(\nu-(j-1)\frac{d}{2}\right)_{m_j},\label{eq:PochhammerFormula}
\end{equation}
where $(\alpha)_n=\alpha(\alpha+1)\cdots(\alpha+n-1)$ denotes the classical Pochhammer symbol.

\subsection{The Wallach set}\label{sec:WallachSet}

We define the generalized Riesz distributions $R_\nu$ on $V$ by
$$ \langle R_\nu,\varphi\rangle = \frac{2^n}{\Gamma_\Omega(\nu)}\int_\Omega\varphi(u)\Delta(2u)^{\nu-\frac{n}{r}}\,du \qquad (\varphi\in\calS(V)). $$
By \cite[Theorems VII.2.2 and VII.3.1]{FK94}, $R_\nu$ is a nowhere vanishing holomorphic family of distributions and it is a positive measure if and only if $\nu$ is contained in the so-called \emph{Wallach set}
$$ \calW = \left\{0,\frac{d}{2},\ldots,(r-1)\frac{d}{2}\right\}\cup\left((r-1)\frac{d}{2},\infty\right). $$
For $\nu\in((r-1)\frac{d}{2},\infty)$, the support of $R_\nu$ is $\overline{\Omega}$ while for $\nu=k\frac{d}{2}$ the support is the subset of the boundary of $\Omega$ consisting of elements of rank at most $k$.

\subsection{Generalized Laguerre polynomials and Laguerre functions}\label{sec:GenLaguerre}

For every ${\bf m}\geq0$ the polynomial $\Phi_{\bf m}(e+x)$ is $(K\cap L)$-invariant and contained in $\bigoplus_{|{\bf n}|\leq|{\bf m}|}\calP_{\bf n}(V)$. We can therefore write
$$ \Phi_{\bf m}(e+x) = \sum_{|{\bf n}|\leq|{\bf m}|}{{\bf m}\choose{\bf n}}\Phi_{\bf n}(x) $$
for some coefficients ${{\bf m}\choose{\bf n}}\in\CC$ called \emph{generalized binomial coefficients}. Using these, we define the \emph{generalized Laguerre polynomials}, see \cite[p. 343]{FK94}
$$ L_{\bf m}^\nu(x) = (\nu)_{\bf m}\sum_{|{\bf n}|\leq|{\bf m}|}{{\bf m}\choose{\bf n}}\frac{1}{(\nu)_{\bf n}}\Phi_{\bf n}(-x) $$
and the \emph{generalized Laguerre functions}
$$ \ell_{\bf m}^\nu(x) = e^{-\tr(x)}L_{\bf m}^\nu(2x). $$

\begin{remark}\label{rem:DefLaguerre}
	For $r=1$ we have $L_{\bf m}^\nu(x)=m_1!\cdot L_{m_1}^{\nu-1}(x)$, where the classical Laguerre polynomial $L_n^\alpha(x)$ is defined by
	$$ L_n^\alpha(x) = \sum_{k=0}^n(-1)^k{n+\alpha\choose n-k}\frac{x^k}{k!}. $$
\end{remark}

\subsection{Generalized Bessel functions}
Denote by $g^*$ the adjoint of $g\in L_\mathbb {C}$ with respect to the inner product $(z,w)\mapsto(z|\overline{w})$ on $V_\CC$. There exists a unique positive definite polynomial $\Phi_{\bf m}(z,w)$ on $V_\CC\times V_\CC$ which is holomorphic in $z$ and antiholomorphic in $w$ such that 
\[\Phi_{\bf m}(gz,w)=\Phi_{\bf m}(z,g^*w)\quad\text{for all } g\in L_\CC\]
and
\[\Phi_{\bf m}(z,e)=\Phi_{\bf m}(z),\]
see \cite[Prop. XII.2.4]{FK94}. The generalized $I$-Bessel function is defined by
$$ \calI_\nu(z,w) = \sum_{\bf m}\frac{d_{\bf m}}{(\frac{n}{r})_{\bf m}(\nu)_{\bf m}}\Phi_{\bf m}(z,w) \qquad (z,w\in V_\CC). $$
The sum converges for $\nu>(r-1)\frac{d}{2}$ and all $z,w\in V_\CC$, and one can extend the definition to $\nu=k\frac{d}{2}\in\calW$ by assuming $z,w$ to be of rank at most $k$. A special case is the one-variable Bessel function
$$ \calI_\nu(z) = \calI_\nu(z,e) = \sum_{\bf m}\frac{d_{\bf m}}{(\frac{n}{r})_{\bf m}(\nu)_{\bf m}}\Phi_{\bf m}(z) \qquad (z\in V_\CC). $$

We further define the generalized $K$-Bessel function on $\Omega$ by the convergent integral
$$ \calK_\nu(x) = \int_\Omega e^{-\tr(u^{-1})-(x|u)}\Delta(u)^{\nu-\frac{2n}{r}}\,du = \int_\Omega e^{-\tr(v)-(x|v^{-1})}\Delta(v)^{-\nu}\,dv. $$
For $\nu=k\frac{d}{2}\in\calW$, the $K$-Bessel function has a well-defined restriction to the boundary orbit consisting of elements of rank $k$ (see \cite[Proposition 3.10]{Moe13}).

\begin{remark} Proposition XII.2.4 in \cite{FK94} shows more than mentioned above. In fact the function
$d_{\bf m} \Phi_{\bf m} (z,w)$ is the reproducing kernel for the finite dimensional space $\calP_{\bf m}$ with respect to a certain $L^2$-inner product.
\end{remark}

\section{Scalar type holomorphic discrete series representations}\label{sec:ScalarHolDS}

The universal cover $\widetilde{G}$ of $G$ has a family of irreducible unitary representations called \emph{scalar type holomorphic discrete series representations}, for which we now recall four different realizations (see \cite[Chapters XIII.1 and XV.4]{FK94} and \cite[Sections 2, 4 and 5]{Moe13} for details). All realizations can be extended holomorphically in the parameter, giving rise to a slightly larger family of representations, the \emph{scalar type unitary highest weight representations}.

\subsection{The tube domain realization}

For $\nu>\frac{2n}{r}-1$ we let
$$ \calH^2_\nu(T_\Omega) = \left\{F\in\calO(T_\Omega):\|F\|_{\calH^2_\nu(T_\Omega)}^2=d_\nu\int_{T_\Omega}|F(z)|^2\Delta(y)^{\nu-\frac{2n}{r}}\,dx\,dy<\infty\right\}, $$
where
$$ d_\nu = \frac{1}{(4\pi)^n}\frac{\Gamma_\Omega(\nu)}{\Gamma_\Omega(\nu-\frac{n}{r})}, $$
and introduce a family $(T_\nu,\calH^2_\nu(T_\Omega))$ of representations of $\widetilde{G}$ by
$$ T_\nu(g)F(z) = \mu_\nu(g^{-1},z)F(g^{-1}z), $$
where $\mu_\nu(g,z)=\chi(\mu(g,z))^{-\frac{\nu}{2}}=(\det\mu(g,z))^{-\frac{r\nu}{2n}}$ with
$$ \mu(g,z) = \left(\frac{\partial(g\cdot z)}{\partial z}\right)^{-1}. $$
This action can be made explicit for the generators of $G$:
\begin{align*}
	T_\nu(n_u)F(z) &= F(z-u) && u\in V,\\
	T_\nu(g)F(z) &= \chi(g)^{-\frac{\nu}{2}}F(g^{-1}z), && g\in L,\\
	T_\nu(j)F(z) &= \Delta(z)^{-\nu}F(-z^{-1}).
\end{align*}

The restriction of $T_\nu$ to $K$ decomposes  into a multiplicity-free direct sum of $K$-types
$$ \calH_\nu^2(T_\Omega) = \bigoplus_{{\bf m}\geq0}\calH_\nu^2(T_\Omega)_{\bf m}. $$
In each $K$-type $\calH_\nu^2(T_\Omega)_{\bf m}$, the subspace of $(K\cap L)$-invariant vectors is one-dimensional and spanned by
\begin{equation}
	\Psi_{\bf m}^\nu(z) = \Delta\left(\frac{z+ie}{2i}\right)^{-\nu}\Phi_{\bf m}\left(\frac{z-ie}{z+ie}\right),\label{eq:DefPsi}
\end{equation}
and the norm of $\Psi_{\bf m}^\nu$ is given by
$$ \|\Psi_{\bf m}^\nu\|_{\calH_\nu^2(T_\Omega)}^2 = \frac{(\frac{n}{r})_{\bf m}}{d_{\bf m}(\nu)_{\bf m}}. $$

\subsection{The bounded domain realization}

For $\nu>\frac{2n}{r}-1$ let
$$ \calH_\nu^2(D) = \left\{f\in\calO(D):\|f\|_{\calH_\nu^2(D)}^2=c_\nu\int_D|f(w)|^2h(w)^{\nu-\frac{2n}{r}}\,dw<\infty\right\}. $$
Here, $h(w)=h(w,w)$ with $h(z,w)$ the unique polynomial holomorphic in $z\in V_\CC$ and antiholomorphic in $w\in V_\CC$ such that $h(x,x)=\Delta(e-x^2)$ for $x\in V$, and the constant $c_\nu$ is chosen such that the constant function $f(w)=1$ has norm $1$:
$$ c_\nu = \frac{1}{\pi^n}\frac{\Gamma_\Omega(\nu)}{\Gamma_\Omega(\nu-\frac{n}{r})}. $$

The Cayley transform
$$ \gamma_\nu:\calH_\nu^2(T_\Omega)\to\calH_\nu^2(D),\quad \gamma_\nu F(w) = \Delta(e-w)^{-\nu}F\left(i\frac{e+w}{e-w}\right) $$
is an isometric isomorphism with inverse (see \cite[Proposition XIII.1.3]{FK94})
$$ \gamma_\nu^{-1}f(z) = \Delta\left(\frac{z+ie}{2i}\right)^{-\nu}f\left(\frac{z-ie}{z+ie}\right). $$
This gives rise to a model $D_\nu$ of the representation $T_\nu$ on $\calH_\nu^2(D)$ defined by
$$ D_\nu(g) = \gamma_\nu\circ T_\nu(g)\circ\gamma_\nu^{-1} \qquad (g\in\widetilde{G}). $$

The decomposition of $D_\nu$ into $K$-types will be written as
$$ \calH_\nu^2(D) = \bigoplus_{{\bf m}\geq0}\calH_\nu^2(D)_{\bf m}, $$
where $\calH_\nu^2(D)_{\bf m}=\calP_{\bf m}$. Hence, the subspace of $(K\cap L)$-invariant vectors in $\calH_\nu^2(D)_{\bf m}$ is spanned by $\Phi_{\bf m}=\gamma_\nu\Psi_{\bf m}^\nu$ with norm (see \cite[Proposition XIII.2.2]{FK94})
\begin{equation}
	\|\Phi_{\bf m}\|^2_{\calH_\nu^2(D)} = \frac{(\frac{n}{r})_{\bf m}}{d_{\bf m}(\nu)_{\bf m}}.\label{eq:NormSphericalPolynomials}
\end{equation}

\subsection{The $L^2$-model}

For $\nu>\frac{n}{r}-1$ let
$$ L^2_\nu(\Omega) = L^2(\Omega,\Delta(2u)^{\nu-\frac{n}{r}}\,du), $$
equipped with the norm
$$ \|\varphi\|_{L^2_\nu(\Omega)}^2 = \frac{2^n}{\Gamma_\Omega(\nu)}\int_\Omega|\varphi(u)|^2\Delta(2u)^{\nu-\frac{n}{r}}\,du. $$
By \cite[Theorem XIII.1.1]{FK94} the Laplace transform
$$ \calL_\nu:L^2_\nu(\Omega)\to\calH_\nu^2(T_\Omega), \quad \calL_\nu\varphi(z) = 
\frac{2^n}{\Gamma_\Omega(\nu)}\int_\Omega e^{i(z|u)}\varphi(u)\Delta(2u)^{\nu-\frac{n}{r}}\,du $$
is an isometric isomorphism. It can be used to realize the representation $T_\nu$ on $L^2_\nu(\Omega)$:
$$ L_\nu(g) = \calL_\nu^{-1}\circ T_\nu(g)\circ\calL_\nu \qquad (g\in\widetilde{G}). $$

In this realization, the statements of Section~\ref{sec:WallachSet} can be used to extend the family $L_\nu$ to all $\nu\in\calW$. This is the \emph{analytic continuation of the scalar type holomorphic discrete series}, or the \emph{scalar type unitary highest weight representations}. Using Laplace and Cayley transform, the realizations $\calH_\nu^2(T_\Omega)$ and $\calH_\nu^2(D)$ can also be defined for $\nu\in\calW$.

The subspace $L^2_\nu(\Omega)^{K\cap L}$ of $(K\cap L)$-invariant vectors is spanned by the 
Laguerre functions $\ell_{\bf m}^\nu$, and by \cite[Proposition XV.4.2]{FK94} we have
\begin{equation}
	\calL_\nu\ell_{\bf m}^\nu = (\nu)_{\bf m}\Psi_{\bf m}^\nu.\label{eq:LaplaceOfLaguerre}
\end{equation}
Hence, their norm is given by (see \cite[Corollary XV.4.2]{FK94})
\begin{equation}
	\|\ell_{\bf m}^\nu\|_{L^2_\nu(\Omega)}^2 = \frac{(\frac{n}{r})_{\bf m}(\nu)_{\bf m}}{d_{\bf m}}.\label{eq:NormLaguerreFunctions}
\end{equation}

\subsection{The Fock model}

In \cite{Moe13} yet another model $F_\nu$ of the scalar type holomorphic discrete series was constructed, the \emph{Fock model}. It is realized on the Hilbert space
$$ \calF_\nu^2(V_\CC) = \left\{F\in\calO(V_\CC):\frac{1}{2^{3r\nu}\Gamma_\Omega(\frac{n}{r})}\int_{V_\CC}|F(z)|^2\omega_\nu(z)\,d\mu_\nu(z)<\infty\right\}, $$
where $\omega_\nu(ux^{\frac{1}{2}})=\calK_\nu(\frac{x}{4})$ ($u\in U$, $x\in\Omega$) with $\calK_\nu$ the $K$-Bessel function on $\Omega$, and
$$ \int_{V_\CC}\varphi(z)\,d\mu_\nu(z) = \frac{2^n}{\Gamma_\Omega(\nu)}\int_U\int_\Omega\varphi(ux^{\frac{1}{2}})\Delta(2x)^{\nu-\frac{n}{r}}\,dx\,du. $$
Here, $x^{\frac{1}{2}}$ denotes the unique square root of an element $x\in\Omega$.

The unitary intertwining operator $\BB_\nu:L^2_\nu(\Omega)\to\calF_\nu^2(V_\CC)$ is the Segal--Bargmann transform
$$ \BB_\nu\varphi(z) = \frac{2^n}{\Gamma_\Omega(\nu)}e^{-\frac{1}{2}\tr(z)}\int_\Omega\calI_\nu(z,x)e^{-\tr(x)}\varphi(x)\Delta(2x)^{\nu-\frac{n}{r}}\,dx, $$
where $\calI_\nu(z,w)$ denotes the $I$-Bessel function. It satisfies
\begin{equation}
	\BB_\nu\ell_{\bf m}^\nu = \left(-\frac{1}{2}\right)^{|{\bf m}|}\Phi_{\bf m}.\label{eq:SBonLaguerre}
\end{equation}

\section{The generating function for the Laguerre polynomials}\label{sec:GenFctLaguerre}

Recall the spherical polynomials $\Phi_{\bf m}(w)$ and the Laguerre functions $\ell_{\bf m}^\nu(u)$. In this section we give two different proofs for the following generating function formula:

\begin{theorem}
	For $w\in D$, $u\in\overline{\Omega}$ and $\nu>(r-1)\frac{d}{2}$ we have
	\begin{equation}
		\sum_{{\bf m}\geq0}\frac{d_{\bf m}}{(\frac{n}{r})_{\bf m}}\Phi_{\bf m}(w)\ell_{\bf m}^\nu(u) = \Delta(e-w)^{-\nu}\int_{K\cap L} e^{-(ku|(e+w)(e-w)^{-1})}\,dk.\label{eq:GenFctLaguerreFct}
	\end{equation}
	This identity is still valid for $\nu=k\frac{d}{2}\in\calW$ if $u\in\partial\Omega$ is of rank at most $k$.
\end{theorem}

Writing $\ell_{\bf m}^\nu(u) = e^{-\tr(u)}L_{\bf m}^\nu(2u)$ and using
\begin{align*}
e^{-(ku|(e+w)(e-w)^{-1})}  & = e^{-(ku|e+2w(e-w)^{-1})} \\
&= e^{-\tr(u)}e^{-2(ku|2w(e-w)^{-1})}, 
\end{align*}
we obtain the following equivalent formula which was stated in \cite[Exercise 3 in Chapter XV]{FK94}:
\begin{equation}
	\sum_{{\bf m}\geq0}\frac{d_{\bf m}}{(\frac{n}{r})_{\bf m}}\Phi_{\bf m}(x)L_{\bf m}^\nu(u) = \Delta(e-x)^{-\nu}\int_{K\cap L} e^{-(ku|x(e-x)^{-1})}\,dk,\label{eq:GenFctLaguerrePoly}
\end{equation}

For $V=\RR$ we have $K\cap L=\{\id\}$ and we recover the generating function of the classical Laguerre polynomials~\eqref{eq:GenFctClassicalLaguerre}.

\subsection{First proof á la Faraut--Koranyi}\label{sec:FirstProof}

Following \cite[Exercise 3 in Chapter XV]{FK94}, we first show
\begin{equation}
	\sum_{{\bf m}\geq0}\frac{d_{\bf m}(\nu)_{\bf m}}{(\frac{n}{r})_{\bf m}}\Phi_{\bf m}(x)\Phi_{\bf m}(y) = \Delta(x)^{-\nu}\int_K\Delta(kx^{-1}-y)^{-\nu}\,dk.\label{eq:FKExercise1}
\end{equation}
For this, note that for fixed invertible $x\in V_\CC$, the right hand side of \eqref{eq:FKExercise1} defines a $K$-invariant analytic function $F_x(y)$. Hence, by \cite[Chapter XII.1]{FK94} it has a power series expansion
$$ F_x(y) = \sum_{{\bf m}\geq0}\frac{d_{\bf m}}{(\frac{n}{r})_{\bf m}}a_{\bf m}\Phi_{\bf m}(y) \qquad \mbox{with} \qquad a_{\bf m} = \left(\Phi_{\bf m}\left(\frac{\partial}{\partial y}\right)F_x\right)(0). $$
Using \cite[Corollary VII.1.3]{FK94} we can write
$$ F_x(y) = \frac{1}{\Gamma_\Omega(\nu)}\Delta(x)^{-\nu}\int_K\int_\Omega e^{-(kx^{-1}-y|u)}\Delta(u)^{\nu-\frac{n}{r}}\,du\,dk $$
and hence
\begin{align*}
	a_{\bf m} &= \frac{1}{\Gamma_\Omega(\nu)}\Delta(x)^{-\nu}\int_K\int_\Omega\Phi_{\bf m}(u)e^{-(kx^{-1}|u)}\Delta(u)^{\nu-\frac{n}{r}}\,du\,dk\\
	&= \frac{1}{\Gamma_\Omega(\nu)}\Delta(x)^{-\nu}\int_K\Gamma_\Omega({\bf m}+\nu)\Delta(kx^{-1})^{-\nu}\Phi_{\bf m}((kx^{-1})^{-1})\,dk
\end{align*}
by \cite[Lemma XI.2.3]{FK94}. Now,
$$ \Delta(kx^{-1})=\Delta(x^{-1})=\Delta(x)^{-1} \qquad \mbox{and} \qquad \Phi_{\bf m}((kx^{-1})^{-1})=\Phi_{\bf m}(x), $$
so
$$ a_{\bf m}=\frac{\Gamma_\Omega({\bf m}+\nu)}{\Gamma_\Omega(\nu)}\Phi_{\bf m}(x)=(\nu)_{\bf m}\Phi_{\bf m}(x) $$
and \eqref{eq:FKExercise1} follows. Next, we note that by \cite[Proposition XV.4.2]{FK94}:
\begin{equation}
	\int_\Omega e^{-(u|y)}L_{\bf m}^\nu(u)\Delta(u)^{\nu-\frac{n}{r}}\,du = \Gamma_\Omega({\bf m}+\nu)\Delta(y)^{-\nu}\Phi_{\bf m}(e-y^{-1}).\label{eq:FKExercise2}
\end{equation}
Now we finally show \eqref{eq:GenFctLaguerrePoly} by comparing the Laplace transform in $u$ of both sides. For the left hand side we obtain with \eqref{eq:FKExercise1} and \eqref{eq:FKExercise2}:
\begin{align*}
	& \int_\Omega e^{-(u|y)}\sum_{{\bf m}\geq0}\frac{d_{\bf m}}{(\frac{n}{r})_{\bf m}}\Phi_{\bf m}(x)L_{\bf m}^\nu(u)\Delta(u)^{\nu-\frac{n}{r}}\,du\\
	&\qquad= \sum_{{\bf m}\geq0}\frac{d_{\bf m}}{(\frac{n}{r})_{\bf m}}\Phi_{\bf m}(x)\Gamma_\Omega({\bf m}+\nu)\Delta(y)^{-\nu}\Phi_{\bf m}(e-y^{-1})\\
	&\qquad= \Gamma_\Omega(\nu)\Delta(y)^{-\nu}\sum_{{\bf m}\geq0}\frac{d_{\bf m}(\nu)_{\bf m}}{(\frac{n}{r})_{\bf m}}\Phi_{\bf m}(x)\Phi_{\bf m}(e-y^{-1})\\
	&\qquad= \Gamma_\Omega(\nu)\Delta(y)^{-\nu}\Delta(x)^{-\nu}\int_K\Delta(kx^{-1}-(e-y^{-1}))^{-\nu}\,dk.
\end{align*}
The Laplace transform of the right hand side of \eqref{eq:GenFctLaguerrePoly} can be evaluated using \cite[Corollary VII.1.3]{FK94}:
\begin{align*}
	& \int_\Omega e^{-(u|y)}\Delta(e-x)^{-\nu}\int_K e^{-(ku|x(e-x)^{-1})}\,dk\Delta(u)^{\nu-\frac{n}{r}}\,du\\
	&\qquad= \Delta(e-x)^{-\nu}\int_K\int_\Omega e^{-(u|y+k^{-1}(x(e-x)^{-1}))}\Delta(u)^{\nu-\frac{n}{r}}\,dy\,dk\\
	&\qquad= \Gamma_\Omega(\nu)\Delta(e-x)^{-\nu}\int_K\Delta(y+k^{-1}(x(e-x)^{-1}))^{-\nu}\,dk.
\end{align*}
That these two expressions are in fact the same, follows from the classical formula (see \cite[Lemma X.4.4]{FK94})
\begin{equation*}
	\Delta(a^{-1}+b^{-1})=\Delta(a)^{-1}\Delta(b)^{-1}\Delta(a+b).\hfill\qed
\end{equation*}

\subsection{Second proof using representation theory}\label{sec:SecondProof}

The composition $\gamma_\nu\circ\calL_\nu$ of the Laplace transform and the Cayley transform is an isometric isomorphism $L^2_\nu(\Omega)\to\calH_\nu^2(D)$ intertwining the action of $G$. Hence, it restricts to an isomorphism $L^2_\nu(\Omega)^{K\cap L}\to\calH_\nu^2(D)^{K\cap L}$ between the $(K\cap L)$-invariant vectors. By \eqref{eq:LaplaceOfLaguerre} and \eqref{eq:DefPsi} we have
$$ \gamma_\nu\circ\calL_\nu\ell_{\bf m}^\nu = (\nu)_{\bf m}\Phi_{\bf m}, $$
so $\gamma_\nu\circ\calL_\nu$ maps the orthogonal basis $(\ell_{\bf m}^\nu)_{\bf m}$ of $L^2_\nu(\Omega)^{K\cap L}$ to the orthogonal basis $((\nu)_{\bf m}\Phi_{\bf m})_{\bf m}$. It follows from \eqref{eq:NormLaguerreFunctions} that $\gamma_\nu\circ\calL_\nu:L^2_\nu(\Omega)^{K\cap L}\to\calH_\nu^2(D)^{K\cap L}$ is given by the integral kernel
\begin{equation}
	K(w,u) = \sum_{\bf m}(\nu)_{\bf m}\Phi_{\bf m}(w)\frac{\ell_{\bf m}^\nu(u)}{\|\ell_{\bf m}^\nu\|_{L^2_\nu(\Omega)}^2} = \sum_{\bf m}\frac{d_{\bf m}}{(\frac{n}{r})_{\bf m}}\Phi_{\bf m}(w)\ell_{\bf m}^\nu(u).\label{eq:KernelCayleyLaplace1}
\end{equation}

On the other hand, for $\varphi\in L^2_\nu(\Omega)^{K\cap L}$ we have by definition
\begin{align*}
	& \gamma_\nu\circ\calL_\nu\varphi(w) = \Delta(e-w)^{-\nu}\calL_\nu\varphi\left(i\frac{e+w}{e-w}\right)\\
	={}& \frac{2^n}{\Gamma_\Omega(\nu)}\Delta(e-w)^{-\nu}\int_\Omega e^{-(u|(e+w)(e-w)^{-1})}\varphi(u)\Delta(2u)^{\nu-\frac{n}{r}}\,du\\
	={}& \frac{2^n}{\Gamma_\Omega(\nu)}\int_\Omega\left(\Delta(e-w)^{-\nu}\int_{K\cap L} e^{-(ku|(e+w)(e-w)^{-1})}\,dk\right)\varphi(u)\Delta(2u)^{\nu-\frac{n}{r}}\,du,
\end{align*}
where we have used that $\Delta(kx)=\Delta(x)$ for all $k\in K\cap L$. It follows that $\gamma_\nu\circ\calL_\nu$ is on the subspace of $(K\cap L)$-invariant vectors given by the integral kernel
\begin{equation}
	K(w,u) = \Delta(e-w)^{-\nu}\int_{K\cap L} e^{-(ku|(e+w)(e-w)^{-1})}\,dk.\label{eq:KernelCayleyLaplace2}
\end{equation}
Comparing \eqref{eq:KernelCayleyLaplace2} with \eqref{eq:KernelCayleyLaplace1} shows the claim.

\subsection{Applying the second proof to the Fock model}\label{sec:SecondProofVar}

Let us carry out the same computations as in Section~\ref{sec:SecondProof} for the Segal--Bargmann transform $\BB_\nu:L^2_\nu(\Omega)\to\calF^2_\nu(V_\CC)$. Restricting the Segal--Bargmann transform to the subspace of $(K\cap L)$-invariant vectors yields an isometric isomorphism $\BB_\nu:L^2_\nu(\Omega)^{K\cap L}\to\calF^2_\nu(V_\CC)^{K\cap L}$. By \eqref{eq:SBonLaguerre} it maps the orthogonal basis $(\ell_{\bf m}^\nu)_{\bf m}$ of $L^2_\nu(\Omega)^{K\cap L}$ to the orthogonal basis $((-\frac{1}{2})^{|{\bf m}|}\Phi_{\bf m})_{\bf m}$ of $\calF^2_\nu(V_\CC)^{K\cap L}$, so its integral kernel is given by
\begin{align*}
 K(z,x)  &= \sum_{\bf m}\left(-\frac{1}{2}\right)^{|{\bf m}|}\Phi_{\bf m}(z)\frac{\ell_{\bf m}^\nu(x)}{\|\ell_{\bf m}^\nu\|_{L^2_\nu(\Omega)}^2}\\
 & = \sum_{\bf m}\frac{(-1)^{|{\bf m}|}d_{\bf m}}{2^{|{\bf m}|}(\frac{n}{r})_{\bf m}(\nu)_{\bf m}}\Phi_{\bf m}(-z)\ell_{\bf m}^\nu(x). 
 \end{align*}

On the other hand, for $\varphi\in L^2_\nu(\Omega)^{K\cap L}$ we have by definition
\begin{align*}
	\BB_\nu\varphi(z) &= \frac{2^n}{\Gamma_\Omega(\nu)}e^{-\frac{1}{2}\tr(z)}\int_\Omega\calI_\nu(z,x)e^{-\tr(x)}\varphi(x)\Delta(2x)^{\nu-\frac{n}{r}}\,dx\\
	&= \frac{2^n}{\Gamma_\Omega(\nu)}e^{-\frac{1}{2}\tr(z)}\int_\Omega\left(\int_{K\cap L}\calI_\nu(z,kx)\,dk\right)e^{-\tr(x)}\varphi(x)\Delta(2x)^{\nu-\frac{n}{r}}\,dx.
\end{align*}
It follows that
$$ K(z,x) = e^{-\frac{1}{2}\tr(z)-\tr(x)}\int_{K\cap L}\calI_\nu(z,kx)\,dk. $$

This shows:

\begin{theorem}
	For $z\in V_\CC$, $x\in\overline{\Omega}$ and $\nu>(r-1)\frac{d}{2}$ we have
	$$ \sum_{\bf m}\frac{(-1)^{|{\bf m}|}d_{\bf m}}{2^{|{\bf m}|}(\frac{n}{r})_{\bf m}(\nu)_{\bf m}}\Phi_{\bf m}(z)L_{\bf m}^\nu(2x) = e^{-\frac{1}{2}\tr(z)}\int_{K\cap L}\calI_\nu(z,kx)\,dk. $$
	This identity is still valid for $\nu=k\frac{d}{2}\in\calW$ if $x\in\partial\Omega$ is of rank at most $k$.
\end{theorem}

\begin{remark}\label{rem:FockModelKernelExpansionSpecialCases}
	Two special cases of this identity are:
	\begin{itemize}
		\item For $x=te$, $t\in\RR$:
		$$ \sum_{\bf m}\frac{(-1)^{|{\bf m}|}d_{\bf m}L_{\bf m}^\nu(2te)}{2^{|{\bf m}|}(\frac{n}{r})_{\bf m}(\nu)_{\bf m}}\Phi_{\bf m}(z) = e^{-\frac{1}{2}\tr(z)}\calI_\nu(tz). $$
		This is the series expansion of the $(K\cap L)$-invariant function $z\mapsto e^{-\frac{1}{2}\tr(z)}\calI_\nu(tz)$ into the spherical polynomials $\Phi_{\bf m}(z)$ which form a basis for the $(K\cap L)$-invariant polynomials. For $t=0$, we obtain the expansion of $e^{\tr(w)}$ into spherical polynomials (see \cite[Proposition XII.1.3~(i)]{FK94}).
		\item For $z=t e$, $t\in\CC$:
		$$ \sum_{\bf m}\frac{(-t)^{|{\bf m}|}e^{\frac{1}{2}\tr (tx) }d_{\bf m}}{2^{|{\bf m}|}(\frac{n}{r})_{\bf m}(\nu)_{\bf m}}L_{\bf m}^\nu(2x) = \calI_\nu(tx). $$
		This is the series expansion of the $(K\cap L)$-invariant function $\calI_\nu(tx)$ into the polynomials $L_{\bf m}^\nu(2x)$ which also form a basis for the $(K\cap L)$-invariant polynomials.
	\end{itemize}
\end{remark}

\section{$K$-type expansion of Whittaker vectors}\label{sec:KTypeExpansion}

For an irreducible unitary representation $(\pi,\calH)$ of $G$ let $\calH^\infty$ denote the subspace of smooth vectors and let $\calH^{-\infty}=\overline{\calH^\infty}'$ be its conjugate dual, the space of distribution vectors. We embed $\calH$ into $\calH^{-\infty}$ by
$$ \iota:\calH\hookrightarrow\calH^{-\infty}, \quad \iota(v)(\overline{\varphi}) = \langle v,\varphi\rangle \qquad (\varphi\in\calH^\infty). $$
The group $G$ acts on $\calH^\infty$ by restriction of $\pi$, i.e. $\pi^\infty(g)=\pi(g)|_{\calH^\infty}$, and on $\calH^{-\infty}$ by the conjugate dual representation $\pi^{-\infty}$. (Note that this definition differs from the one in \cite{FOO22}, where no complex conjugate was used. However, all results on Whittaker vectors can be translated canonically by applying complex conjugation since we only consider scalar-valued functions.)

For $v\in\Omega$, we define a generic unitary character $\psi_v$ of $N$ by
$$ \psi_v(n_u) = e^{-i(u|v)} \qquad (u\in V). $$
For an irreducible unitary representation $(\pi,\calH)$ of $G$ we let $\calH^{-\infty,\psi_v}$ denote the space of Whittaker vectors, i.e. the space of all $W\in\calH^{-\infty}$ such that $\pi^{-\infty}(n)W=\psi_v(n)W$.

In \cite{FOO22} we found explicit expressions for the Whittaker vectors on all holomorphic discrete series representations in the first three models described in Section~\ref{sec:ScalarHolDS}. For the case of scalar type holomorphic discrete series, the space of Whittaker vectors is one-dimensional, and in this section we find their $K$-type expansion.

\subsection{The $L^2$-model}

In the $L^2$-model, the space $L^2_\nu(\Omega)^{-\infty}$ can be identified with a space of distributions on $\Omega$ since $C_c^\infty(\Omega)\subseteq L^2_\nu(\Omega)^\infty\subseteq C^\infty(\Omega)$. Under this identification, the Dirac distribution $W_{\nu,v}^\Omega=\delta_v$ given by
$$ W^\Omega_{\nu,v}(\overline{\varphi}) = \overline{\varphi(v)} \qquad (\varphi\in L^2_\nu(\Omega)^\infty) $$
is the unique (up to scalar multiples) Whittaker vector on $L^2_\nu(\Omega)$. (Note that although we assumed $\nu>\frac{2n}{r}-1$ in \cite{FOO22}, the same proof works for $\nu>(r-1)\frac{d}{2}$.)

\begin{theorem}\label{thm:ExpansionDirac}
	For $\nu>(r-1)\frac{d}{2}$ and $t>0$ we have
	\begin{equation}\label{eq:expansion}
		W_{\nu,te}^\Omega = \sum_{\bf m}
		\frac{d_{\bf m}\ell^\nu_{\bf m} (te)}{(\frac{n}{r})_{\bf m}(\nu)_{\bf m}}\ell_{\bf m}^\nu
	\end{equation}
	with convergence in the weak$^*$-topology.
\end{theorem}

\begin{proof}
	The distribution vector $W_{\nu,te}^\Omega=\delta_{te}$ is $(K\cap L)$-invariant and can therefore be expanded as
	\[ W_{\nu,te}^\Omega = \sum_{{\bf m} } a ({\bf m} ,\nu )\frac{\ell_{\bf m}^\nu}{\|\ell_{\bf m}^\nu\|} \]
	for some constants $a({\bf m},\nu)\in\CC$. Since the functions $(\|\ell_{\bf m}^\nu\|^{-1}\ell_{\bf m}^\nu)_{\bf m}$ form an orthonormal basis of $L^2_\nu(\Omega)^{K\cap L}$, it follows that
	\[ a({\bf m} ,\nu) = \left\langle W_{\nu,te}^\Omega,\frac{\ell^\nu_{\bf m} }{\|\ell^\nu_{\bf m}\|}\right\rangle= \frac{\ell_{\bf m}^\nu(te)}{\|\ell^\nu_{\bf m} \|} .\]
	Now the claim follows with \eqref{eq:NormLaguerreFunctions}.
\end{proof}

\begin{remark}
	\begin{enumerate}
		\item Theorem~\ref{thm:ExpansionDirac} does not extend to $\nu=k\frac{d}{2}\in\calW$. The reason is that in this case $(\nu)_{\bf m}=0$ whenever $m_{k+1}>0$. In order to make sense of the right hand side, one would have to view it as a series expansion for functions/distributions on the elements of rank at most $k$ in $\partial\Omega$. For $x$ of rank at most $k$, one has $\ell_{\bf m}^\nu(x)=0$ for $m_{k+1}>0$, so the sum can be viewed as a sum over $m_1\geq\ldots\geq m_k\geq m_{k+1}=\ldots=m_r=0$. However, the left hand side $\delta_{te}$ is not defined as a distribution on the boundary $\partial\Omega$ of $\Omega$ since $te$ is contained in the interior of $\Omega$.
		\item For $t\to0$, \eqref{eq:expansion} becomes
		$$ \delta_0 = \sum_{\bf m}\frac{d_{\bf m}}{(\frac{n}{r})_{\bf m}}\ell_{\bf m}^\nu. $$
		Note that $\delta_0$ is a distribution vector in $L^2_\nu(\Omega)^{-\infty}$ for $\nu>(r-1)\frac{d}{2}$ by \cite[Theorem 6.11]{Moe13}.
	\end{enumerate}
\end{remark}

\subsection{The tube domain model}

In the tube domain model, the space $\calH^2_\nu(T_\Omega)^{-\infty}$ can be identified with a space of holomorphic functions $F$ on $T_\Omega$ by
$$ F(\overline{\varphi}) = d_\nu\int_{T_\Omega}F(z)\overline{\varphi(z)}\Delta(y)^{\nu-\frac{2n}{r}}\,dx\,dy \qquad (\varphi\in\calH_\nu^2(T_\Omega)^\infty). $$
Applying the Laplace transform to $W_{\nu,v}^\Omega=\delta_v$ we obtain that the following holomorphic function on $T_\Omega$ is a Whittaker vector on $\calH_\nu^2(T_\Omega)$:
$$ W_{\nu,v}^{T_\Omega}(z) = \calL_\nu W_{\nu,v}^\Omega(z) = e^{i(z|v)} \qquad (z\in T_\Omega). $$
Applying the Laplace transform to \eqref{eq:expansion} and using \eqref{eq:LaplaceOfLaguerre} yields the $K$-type expansion of $W_{\nu,te}^{T_\Omega}$:

\begin{corollary}
	For $\nu>(r-1)\frac{d}{2}$ and $t>0$ we have
	\begin{equation}\label{eq:tP}
		W_{\nu,te}^{T_\Omega} = \sum_{\bf m}\frac{d_{\bf m}\ell_{\bf m}^\nu(te)}{(\frac{n}{r})_{\bf m}}\Psi_{\bf m}^\nu.
	\end{equation}
\end{corollary}

\subsection{The bounded domain model}

As in the tube domain model, we identify $\calH_\nu^2(D)^{-\infty}$ with a space of holomorphic functions $F$ on $D$ by
$$ F(\overline{\varphi}) = c_\nu\int_D F(w)\overline{\varphi(w)}h(w)^{\nu-\frac{2n}{r}}\,dw \qquad (\varphi\in\calH_\nu^2(D)^\infty). $$
We apply the Cayley transform to $W_{\nu,v}^{T_\Omega}$ to obtain a Whittaker vector on $\calH_\nu^2(D)$:
$$ W_{\nu,v}^D(w) = \gamma_\nu W_{\nu,v}^{T_\Omega}(w) = \Delta(e-w)^{-\nu}e^{-((e+w)(e-w)^{-1}|v)}. $$
Applying the Cayley transform to \eqref{eq:tP} and using $\gamma_\nu\Psi_{\bf m}^\nu=\Phi_{\bf m}$ yields:

\begin{corollary}
	For $\nu>(r-1)\frac{d}{2}$ and $t>0$ we have
	\begin{equation}
		W_{\nu,te}^D = \sum_{\bf m}\frac{d_{\bf m}\ell_{\bf m}^\nu(te)}{(\frac{n}{r})_{\bf m}}\Phi_{\bf m}.\label{eq:WhittakerExpansionD}
	\end{equation}
\end{corollary}

\begin{remark}
	Formula \eqref{eq:WhittakerExpansionD} is a special case of \eqref{eq:GenFctLaguerreFct} for $u=te$.
\end{remark}

\subsection{The Fock model}

As in the tube domain and the bounded domain model, we identify $\calF_\nu^2(V_\CC)^{-\infty}$ with a space of holomorphic functions $F$ on $V_\CC$ by
$$ F(\overline{\varphi}) = \frac{2^n}{\Gamma_\Omega(\nu)} \int_{V_\CC}F(z)\overline{\varphi(z)}\omega_\nu(z)\,d\mu_\nu(z). $$
Applying the Segal--Bargmann transform $\BB_\nu$ to $W_{\nu,v}^\Omega=\delta_v$ we obtain a Whittaker vector on $\calF_\nu^2(V_\CC)$:
$$ W_{\nu,v}^\calF(z) = \BB_\nu W_{\nu,v}^\Omega(z) = e^{-\frac{1}{2}\tr(z)}\calI_\nu(z,v)e^{-\tr(v)}. $$
Note that for $v=te$ this equals
$$ W_{\nu,te}^\calF(z) = e^{-rt}e^{-\frac{1}{2}\tr(z)}\calI_\nu(tz). $$
Applying $\BB_\nu$ to \eqref{eq:expansion} and using \eqref{eq:SBonLaguerre} gives the following expansion:

\begin{corollary}
	For $\nu>(r-1)\frac{d}{2}$ and $t>0$ we have
	$$ W_{\nu,te}^\calF = \sum_{\bf m}\frac{(-1)^{|{\bf m}|}d_{\bf m}\ell^\nu_{\bf m}(te)}{2^{|{\bf m}|}(\frac{n}{r})_{\bf m}(\nu)_{\bf m}}\Phi_{\bf m}. $$
\end{corollary}

Note that this is the same expansion as obtained in Remark~\ref{rem:FockModelKernelExpansionSpecialCases}.

\section{Recurrence relations for the Laguerre functions}\label{sec:RecurrenceRelations}

In this section we discuss refinements of the recurrence relations for the Laguerre polynomials obtained
in \cite[Thm. 5.2]{ADO06} and \cite[Thm. 7.9]{DOZ03} and show that
the Laguerre functions $\ell_{\bf m}^\nu$ satisfy a recurrence relation when multiplied with the $(K\cap L)$-invariant linear form $\tr(x)$:

\begin{proposition}
	The following recurrence relation holds:
	\begin{equation}
		\tr(x)\ell_{\bf m}^\nu(x) = a_{\bf m}^\nu\ell_{\bf m}^\nu(x)+\sum_{j=1}^r\Big(b_{{\bf m},j}^\nu\ell_{{\bf m}+e_j}^\nu(x)+c_{{\bf m},j}^\nu\ell_{{\bf m}-e_j}^\nu(x)\Big)\label{eq:LaguerreRecurrence}
	\end{equation}
	with
	\begin{align*}
		a_{\bf m}^\nu ={}& |{\bf m}|+\frac{r\nu}{2},\\
		b_{{\bf m},j}^\nu ={}& -\frac{1}{2}\prod_{i\neq j}\frac{m_j-m_i+(i-j+1)\frac{d}{2}}{m_j-m_i+(i-j)\frac{d}{2}},\\
		c_{{\bf m},j}^\nu ={}& -\frac{1}{2}\left(\nu+m_j-(j-1)\frac{d}{2}-1\right)\left(m_j+(r-j)\frac{d}{2}\right)\\
		& \hspace{4cm}\times\prod_{i\neq j}\frac{m_j-m_i+(i-j-1)\frac{d}{2}}{m_j-m_i+(i-j)\frac{d}{2}}.
	\end{align*}
\end{proposition}

\begin{proof}
	This is essentially \cite[Theorem 5.2]{ADO06}, the only difference being that the coefficient $c_{{\bf m},j}^\nu$ is stated to be
	$$ c_{{\bf m},j}^\nu = -\frac{1}{2}\left(\nu+m_j-(j-1)\frac{d}{2}-1\right){{\bf m}\choose{\bf m}-e_j}. $$
	By \cite[Theorem 5]{Las90} the binomial coefficient equals
	\begin{equation*}
		{{\bf m}\choose{\bf m}-e_j} = \left(m_j+(r-j)\frac{d}{2}\right)\prod_{i\neq j}\frac{m_j-m_i+(i-j-1)\frac{d}{2}}{m_j-m_i+(i-j)\frac{d}{2}}.\qedhere
	\end{equation*}
\end{proof}

We study how the property of $W=W_{\nu,te}^\Omega\in L^2_\nu(\Omega)^{-\infty}$ being a Whittaker vector is related to the recurrence relation \eqref{eq:LaguerreRecurrence}. In the $L^2$-model, the property of being a Whittaker vector can be formulated in terms of the Lie algebra action. More precisely, $W\in L^2_\nu(\Omega)^{-\infty}$ is a Whittaker vector for the character $\psi_{te}(n_u)=e^{-it\tr(u)}$ if and only if
\begin{equation}
	(x|u)\cdot W (x) = t\tr(u)\cdot W (x) \qquad \mbox{for all }u\in V.\label{eq:WhittakerPropertyL2}
\end{equation}
In particular, for $u=e$:
\begin{equation}
	\tr(x)\cdot W (x)  = rt\cdot W (x).\label{eq:WhittakerPropertyL2Special}
\end{equation}
Applying this to \eqref{eq:expansion} yields
$$ \sum_{\bf m}\frac{d_{\bf m}\ell_{\bf m}^\nu(te)}{(\frac{n}{r})_{\bf m}(\nu)_{\bf m}}\tr(x)\ell_{\bf m}^\nu(x) = \sum_{\bf m}rt\frac{d_{\bf m}\ell_{\bf m}^\nu(te)}{(\frac{n}{r})_{\bf m}(\nu)_{\bf m}}\ell_{\bf m}^\nu(x). $$
Applying \eqref{eq:LaguerreRecurrence} and comparing coefficients of $\ell_{\bf m}^\nu$ shows that
\begin{multline*}
	\frac{d_{\bf m}\ell_{\bf m}^\nu(te)}{(\frac{n}{r})_{\bf m}(\nu)_{\bf m}}a_{\bf m}^\nu+\sum_{j=1}^r\Bigg(\frac{d_{{\bf m}-e_j}\ell_{{\bf m}-e_j}^\nu(te)}{(\frac{n}{r})_{{\bf m}-e_j}(\nu)_{{\bf m}-e_j}}b_{{\bf m}-e_j,j}^\nu+\frac{d_{{\bf m}+e_j}\ell_{{\bf m}+e_j}^\nu(te)}{(\frac{n}{r})_{{\bf m}+e_j}(\nu)_{{\bf m}+e_j}}c_{{\bf m}+e_j,j}^\nu\Bigg)\\
	= rt\frac{d_{\bf m}\ell_{\bf m}^\nu(te)}{(\frac{n}{r})_{\bf m}(\nu)_{\bf m}}.
\end{multline*}
Using \eqref{eq:QuotientDimensions} and \eqref{eq:PochhammerFormula} this can be rewritten as
\begin{equation*}
	a_{\bf m}^\nu\ell_{\bf m}^\nu(te)+\sum_{j=1}^r\left(b_{{\bf m},j}^\nu\ell_{{\bf m}+e_j}^\nu(te)+c_{{\bf m},j}^\nu\ell_{{\bf m}-e_j}^\nu(te)\right)=rt\ell_{\bf m}^\nu(te).
\end{equation*}
This is nothing else than the recurrence relation \eqref{eq:LaguerreRecurrence} at $x=te$. Reading the computation backwards, this shows:

\begin{proposition}
	The recurrence relation \eqref{eq:LaguerreRecurrence} implies \eqref{eq:WhittakerPropertyL2Special}.
\end{proposition}

\begin{remark}
	We note that for $r=1$, \eqref{eq:WhittakerPropertyL2} and \eqref{eq:WhittakerPropertyL2Special} are actually equivalent, and the previous discussion shows that the property of $W_{\nu,te}^\Omega$ being a Whittaker vector is equivalent to the recurrence relation \eqref{eq:LaguerreRecurrence}.
\end{remark}


\begin{thebibliography}{99}
	
	\bibitem{AAR99}
	George~E. Andrews, Richard Askey, and Ranjan Roy, \emph{Special functions},
	Encyclopedia of Mathematics and its Applications, vol.~71, Cambridge
	University Press, Cambridge, 1999.
	
	\bibitem{ADO06}
	Michael Aristidou, Mark Davidson, and Gestur \'{O}lafsson, \emph{Differential
		recursion relations for {L}aguerre functions on symmetric cones}, Bull. Sci.
	Math. \textbf{130} (2006), no.~3, 246--263.
	
	\bibitem{DOZ03} Mark Davidson, Gestur \'Olafsson and Genkai Zhang, \emph{Laplace
	and Segal-Bargmann transforms on Hermitian symmetric spaces and othogonal polynomials},
	J. Funct. Anal. \textbf{204} (2003), 157--195
	\bibitem{FK94}
	Jacques Faraut and Adam Kor\'{a}nyi, \emph{Analysis on symmetric cones}, Oxford
	Mathematical Monographs, The Clarendon Press, Oxford University Press, New
	York, 1994, Oxford Science Publications.
	
	\bibitem{FOO22}
	Jan Frahm, Gestur \'{O}lafsson, and Bent {\O}rsted, \emph{The holomorphic
		discrete series contribution to the generalized {W}hittaker {P}lancherel
		formula},  (2022), preprint, available at
	\href{https://arxiv.org/abs/2203.14784}{arXiv:2203.14784}.
	
	\bibitem{Kos00}
	Bertram Kostant, \emph{On {L}aguerre polynomials, {B}essel functions, {H}ankel
		transform and a series in the unitary dual of the simply-connected covering
		group of {${\rm Sl}(2,{\bf R})$}}, Represent. Theory \textbf{4} (2000),
	181--224.
	
	\bibitem{Las90}
	Michel Lassalle, \emph{Une formule du bin\^{o}me g\'{e}n\'{e}ralis\'{e}e pour
		les polyn\^{o}mes de {J}ack}, C. R. Acad. Sci. Paris S\'{e}r. I Math.
	\textbf{310} (1990), no.~5, 253--256.
	
	\bibitem{Moe13}
	Jan M\"{o}llers, \emph{A geometric quantization of the {K}ostant-{S}ekiguchi
		correspondence for scalar type unitary highest weight representations}, Doc.
	Math. \textbf{18} (2013), 785--855.
	
\end{thebibliography}

\providecommand{\bysame}{\leavevmode\hbox to3em{\hrulefill}\thinspace}
\providecommand{\MR}{\relax\ifhmode\unskip\space\fi MR }
\providecommand{\MRhref}[2]{%
	\href{http://www.ams.org/mathscinet-getitem?mr=#1}{#2}
}
\providecommand{\href}[2]{#2}

\end{document}